\documentclass{amsart}
\usepackage{amssymb,amsmath,latexsym}
\usepackage{amsthm}
\usepackage{fontenc}
\usepackage{amssymb}
\numberwithin{equation}{section}

\newtheorem{theorem}{Theorem}[section]
\newtheorem{corollary}{Corollary}[theorem]

\begin{document}
\author{Alexander E. Patkowski}
\title{On the harmonic continuation of the Riemann xi function}

\maketitle
\begin{abstract}We generalize the harmonic continuation of the Riemann xi-function to the $n$-dimension case, to obtain the solution to the Dirichlet problem on $\mathbb{R}_{+}^{n+1}.$ We also provide a new expansion for the harmonic continuation of the Riemann xi-function using an expansion given by R.J. Duffin. \end{abstract}



\section{Introduction and main result}
Recall the Riemann xi-function is given as $\xi(s)=\frac{1}{2}s(s-1)\pi^{-\frac{s}{2}}\Gamma(\frac{s}{2})\zeta(s),$ where $\zeta(s)$ is the Riemann zeta function [11], and $\Gamma(s)$ is the gamma function. In a recent paper [9], we presented an application of classical integrals involving the Riemann xi-function to the Dirichlet problem in the half plane. Let $L(s)$ denote the self-dual principal automorphic $L$-functions and $\bar{\gamma}(s)$ its associated gamma factor, then our main theorem presented therein is given in the following.

\begin{theorem} ([9, Theorem 3.1]) The solution of Dirichlet's problem in the half plane,
$$\frac{\partial^2 w}{\partial y^2}+\frac{\partial^2 w}{\partial x^2}=0,$$
where $y\in\mathbb{R},$ $x\ge0,$ initial condition $w(y,0)=h(y):=\bar{\gamma}(\frac{1}{2}+iy)L(\frac{1}{2}+iy),$ $w(y,x)\rightarrow0,$ as $|y|\rightarrow\infty,$ is given by the Poisson integral. Furthermore, the solution also satisfies the condition $w(0,s-\frac{1}{2})=h(-i(s-\frac{1}{2}))-r(s),$ for an analytic function $r(s)$ that satisfies $h(-i(s-1/2))=C/(s(s-1))+r(s)+r(1-s).$
\end{theorem}
The proof relied on extending a classical integral found in Titchmarsh [11, eq.(2.16.1)]
\begin{equation}\int_{0}^{\infty}\frac{\Xi(t)}{t^2+\frac{1}{4}}\cos(xt)dt=\frac{\pi}{2}\left(e^{x/2}-2e^{-x/2}\psi(e^{-2x})\right),\end{equation}
to self-dual principal automorphic $L$-functions, and then taking the Laplace transform. Here, as usual, $\Xi(x)=\xi(\frac{1}{2}+ix)$ and $\psi(x)=\sum_{n=1}e^{-\pi n^2x}.$ In the special case of the Riemann xi-function, this was offered in [9] as the following corollary. Let $\Gamma(s,x)$ denote the incomplete gamma function over the interval $(x,\infty).$
\begin{corollary} ([9, Corollary 1.2]) \it When $\Re(s)>1,$ the Riemann $\xi$-function satisfies the integral equation,
\begin{equation}\begin{aligned}&\Upsilon(s):=(s-\frac{1}{2})\int_{0}^{\infty}\frac{\Xi(t)}{(t^2+\frac{1}{4})(t^2+(s-\frac{1}{2})^2)}dt\\
&=\frac{\pi}{2}\left(\frac{1}{s-1}-\frac{\xi(s)}{s(s-1)}+\pi^{-s/2}\sum_{n\ge1}n^{-s}\Gamma\left(\frac{s}{2},\pi n^2\right)\right).\end{aligned}\end{equation}
\end{corollary}
Many papers have utilized Fourier integrals of this type to obtain interesting identities and applications [1, 2, 3, 6, 7, 8, 10]. \par The purpose of this paper is to generalize [9, Corollary 1.2] to obtain the harmonic continuation that solves the $n$-dimensional Dirichlet problem. In this way we offer a general result which isn't captured in [9, Theorem 3.1]. We also offer a new expansion of the harmonic continuation of the Riemann xi-function by appealing to the work of R.J. Duffin [4], which will be presented in the last section. To state our main result we will need to define a kernel.  From [5, pg.93, eq.(2.1.13)], given a vector $\mathbf{x}=(x_1,x_2,\cdots, x_l)$ from $\mathbb{R}_{+}^{n},$ we define
$$K(\mathbf{x})=\frac{\Gamma(\frac{n+1}{2})}{\pi^{\frac{n+1}{2}}(1+|\mathbf{x}|^2)^{\frac{n+1}{2}}}.$$

\begin{theorem} Let $g(\mathbf{x})=\prod_{l=1}^{n}\frac{\Xi(x_l)}{x_l^2+\frac{1}{4}}$ and set $K_y(\mathbf{x})=y^{-n}K(y^{-1}\mathbf{x}).$ Then function (convolution) $u(\mathbf{x},y)=(g*K_y)(\mathbf{x})$, solves the Dirichlet problem
$$\partial_y^2 u+\sum_{l=1}^{n}\partial_{x_l}^2u=0,$$
$$u(\mathbf{x},0)=g(\mathbf{x}).$$ Furthermore, we have
$$u(0,y)=(g*K_y)(0)=\int_{\mathbb{R}^n}e^{-2\pi|y\mathbf{z}|}\prod_{l=1}^{n}\left(e^{z_l/2}-2e^{-z_l/2}\psi(e^{-2z_l}) \right)d\mathbf{z}.$$
\end{theorem}
We note that the case $n=1$ reduces to a special case of Theorem 1.1 and contains [9, Corollary 1.2], being that $u(0,y)=(g*K_y)(0)$ essentially becomes the Laplace transform of (1.1).

\section{Proof of the $n$-dimensional case and related comments}
Let $\mathbb{R}_{+}$ denote the positive reals. From [5, pg.92], we see that $u(\mathbf{x},y)=(g*K_y)(\mathbf{x})$ solves the Dirichlet problem
$$\partial_y^2 u+\sum_{l=1}^{n}\partial_{x_l}^2u=0,$$
$$u(\mathbf{x},0)=g(\mathbf{x}).$$
Then we can say that $u(\mathbf{x},y)$ is harmonic on $\mathbb{R}_{+}^{n+1}.$ Recall the $n$-dimensional Fourier transform [5, pg.108] is given by
$$\hat{f}(\eta)=\int_{\mathbb{R}^n}f(\mathbf{x})e^{-2\pi i\mathbf{x}\cdot \mathbf{\eta}}d\mathbf{x},$$
where $\mathbf{x}\cdot \mathbf{\eta}=\sum_{l=1}^nx_l\eta_l.$ Consider $g(\mathbf{x})=\prod_{l=1}^{n}\frac{\Xi(x_l)}{x_l^2+\frac{1}{4}},$ and note [5, pg.112, Theorem 2.2.14]
\begin{equation}\int_{\mathbb{R}^n}f(\mathbf{x})\hat{g}(\mathbf{x})d\mathbf{x}=\int_{\mathbb{R}^n}\hat{f}(\mathbf{x})g(\mathbf{x})d\mathbf{x}.\end{equation}
Invoking [5, pg.118, Exercise 2.2.11],
$$\int_{\mathbb{R}^n}e^{-2\pi|y\mathbf{z}|}e^{-2\pi i\mathbf{z}\cdot \mathbf{x}}d\mathbf{z}=y^{-n}K(y^{-1}\mathbf{x})=K_y(\mathbf{x}),$$ with our chosen $g(\mathbf{x})$ in (2.1) gives $u(0,y),$
$$\int_{\mathbb{R}^n}e^{-2\pi|y\mathbf{z}|}\prod_{l=1}^{n}\left(e^{z_l/2}-2e^{-z_l/2}\psi(e^{-2z_l}) \right)d\mathbf{z}=\int_{\mathbb{R}^n}g(\mathbf{z})K_y(\mathbf{z})d\mathbf{z},$$ since 
\begin{align*}\hat{g}(\mathbf{z})&=\int_{\mathbb{R}^n}g(\mathbf{x})e^{-2\pi i\mathbf{x}\cdot \mathbf{z}}d\mathbf{x}\\
&=\prod_{l=1}^{n}\left(e^{z_l/2}-2e^{-z_l/2}\psi(e^{-2z_l}) \right).
\end{align*}
The right hand side is now precisely $u(0,y)=(g*K_y)(0).$ \par
It is possible to take another direction to obtain a formula which contains [9, Corollary 1.2] as a special case. If we set $z_l=z_j=z$ for $l\neq j,$ and $l,j\le n,$ then we can write
\begin{align}\hat{g}(\mathbf{z})&=\left(e^{z/2}-2e^{-z/2}\psi(e^{-2z}) \right)^n \\
&\sum_{k=0}^{n}\binom{n}{k}e^{z(n-k)/2}\left(-2e^{-z/2}\psi(e^{-2z})\right)^{k}.
\end{align}
We write the sum of squares function to be generated by 
\begin{align*}\sum_{n=0}^{\infty}r_k(n)q^n&=\left(\sum_{n=-\infty}^{\infty}q^{n^2} \right)^k\\
&=\sum_{l=0}^{k}\binom{k}{l}\left( 2\sum_{n=1}^{\infty}q^{n^2}\right)^l\\
&=\sum_{l=0}^{k}\binom{k}{l}2^k\sum_{m=1}^{\infty}r_l'(m)q^m,\end{align*} for $|q|<1$ say, and note that $r_0(0)=1,$ and $r_0(n)=0$ if $n>1.$ Equating coefficients of $q^n$ we find that for $n>1,$ $$r_k(n)=\sum_{l=0}^{k}\binom{k}{l}2^kr_l'(n),$$ which gives us $\frac{1}{2}r_1(n)=r_1'(n).$

Note that taking the Laplace transform of (2.2)--(2.3) similar to the method we employed in [9] gives a different kernel. Namely,
\begin{align*}&\int_{0}^{\infty}\hat{g}(z)e^{-sz}dz\\
&\int_{\mathbb{R}^n}\frac{sg(x)}{s^2+\left(\sum_{l=1}^nx_l \right)^2}dx\\
&=\frac{1}{s-\frac{1}{2}}+\sum_{k=1}^{n}\binom{n}{k}(-2)^k\sum_{m=1}^{\infty}r_k'(m)\int_{0}^{1}t^{s-1+(n-k)/2+k/2}e^{-mt^2}dt\\
&=\frac{1}{s-\frac{1}{2}}\\
&+\frac{1}{2}\sum_{k=1}^{n}\binom{n}{k}(-2)^k\sum_{m=1}^{\infty}r_k'(m)\Gamma\left(\frac{s+(2k-n)/2}{2}\right)\frac{1}{m^{s/2+(2k-n)/4}}\\
&-\frac{1}{2}\sum_{k=1}^{n}\binom{n}{k}(-2)^k\sum_{m=1}^{\infty}r_k'(m)\Gamma\left(\frac{s+(2k-n)/2}{2},\pi m\right)\frac{1}{m^{s/2+(2k-n)/4}},
\end{align*}
which again gives [9, Corollary 1.2] when $n=1.$
\section{The Duffin expansion of the harmonic continuation of the Riemann xi function}

In Duffin's paper [4, pg.275], we find an interesting theorem for the harmonic continuation of a function. As usual, let $\mu(n)$ denote the M$\ddot{o}$bius function [11]. For a definition of $P$ summable see [4, pg.273].

\begin{theorem} ([4, Theorem 2]) Let $f(x)\in L^{1}(0,\infty)$ and let $C(x)$ be its transform relative to the kernel $\cos(2\pi x).$ Then the harmonic continuation is given by
$$f(x,y):=\frac{1}{\pi}\int_{\mathbb{R}}\frac{y}{y^2+(x-t)^2}f(t)dt=\sum_{m=1}^{\infty}\frac{\mu(m)}{m}\sum_{n=1}^{\infty}e^{-2n\pi y/mx}C\left(\frac{n}{mx}\right)$$
where $x, y>0,$ and the sum over $m$ is evaluated by $P$ summability. In particular, $f(x,y)\rightarrow f(x)$ a.e. when $y\rightarrow0.$
\end{theorem}

By the integral (1.1) it readily follows that we have the following result.
\begin{theorem} The harmonic continuation of $$f(x)=\frac{\Xi(x)}{x^2+\frac{1}{4}},$$ is given by
$$f(x,y)=\frac{\pi}{2}\sum_{m=1}^{\infty}\frac{\mu(m)}{m}\sum_{n=1}^{\infty}e^{-2n\pi y/mx}\left(e^{\frac{n}{mx2}}-2e^{-\frac{n}{m2x}}\psi(e^{-2\frac{n}{mx}})\right),$$ where $x, y>0.$ Furthermore,
$$\frac{\Xi(x)}{x^2+\frac{1}{4}}=\frac{\pi}{2}\lim_{y\rightarrow0}\sum_{m=1}^{\infty}\frac{\mu(m)}{m}\sum_{n=1}^{\infty}e^{-2n\pi y/mx}\left(e^{\frac{n}{mx2}}-2e^{-\frac{n}{m2x}}\psi(e^{-2\frac{n}{mx}})\right),$$
and
\begin{align*}&\frac{\pi}{2}\left(\frac{1}{y-\frac{1}{2}}-\frac{\xi(y+\frac{1}{2})(y-\frac{1}{2})}{y+\frac{1}{2}}+\pi^{-(y-\frac{1}{2})/2}\sum_{n\ge1}n^{-y+\frac{1}{2}}\Gamma\left(\frac{y+\frac{1}{2}}{2},\pi n^2\right)\right)\\
&=\frac{\pi}{2}\lim_{x\rightarrow0}\sum_{m=1}^{\infty}\frac{\mu(m)}{m}\sum_{n=1}^{\infty}e^{-2n\pi y/mx}\left(e^{\frac{n}{mx2}}-2e^{-\frac{n}{m2x}}\psi(e^{-2\frac{n}{mx}})\right).\end{align*}
\end{theorem}
\begin{proof} Clearly our choice of $f(x)$ is $L^{1}(0,\infty)$ since $\Xi(x)=O(x^Ne^{-\pi x/4})$ by [11, pg.257]. The second limit case in the theorem, which is all that remains to prove, follows directly from [9, Corollary 1.2].
\end{proof}

A criteria for the Riemann Hypothesis may be formulated from this theorem by putting $x=\gamma$ where $\Im(\rho)=:\gamma$ with $\rho:=\frac{1}{2}+i\gamma,$ the non-trivial zeros of $\zeta(s),$ in the first limit case. Now note that $x>0,$ and non-trivial zeros come in pairs $(\rho,1-\rho).$ It follows that if for every $\gamma>0,$
\begin{equation}\lim_{y\rightarrow0}\sum_{m=1}^{\infty}\frac{\mu(m)}{m}\sum_{n=1}^{\infty}e^{-2n\pi y/m\gamma}\left(e^{\frac{n}{m2\gamma}}-2e^{-\frac{n}{m2\gamma}}\psi(e^{-2\frac{n}{m\gamma}})\right)=0,\end{equation}
then the Riemann hypothesis would be true since we are assuming all the non-trivial zeros have $\Re(s)=\frac{1}{2}.$ That is, if (3.1) is true for every $\gamma>0,$ then it's also true that $\Xi(-\gamma)=0$ and the Riemann Hypothesis follows. The second limit case doesn't appear to offer a simple criteria due to the condition that $y>0.$

1390 Bumps River Rd. \\*
Centerville, MA
02632 \\*
USA \\*
E-mail: alexpatk@hotmail.com, alexepatkowski@gmail.com

\end{document}